\documentclass[10pt]{amsart}
\usepackage{amssymb, amsthm}
\usepackage{setspace}
\usepackage{amsmath}
\usepackage{eucal}
\usepackage{amscd}
\usepackage{url}
\usepackage[all, cmtip]{xy}
\usepackage[pdftex]{ graphicx}
\allowdisplaybreaks[4]

\newtheorem{thm}{Theorem}
\newtheorem{lem}[thm]{Lemma}
\newtheorem{cor}[thm]{Corollary}

\newcommand{\term}[1]{\textit{#1}}
\theoremstyle{definition}
\newtheorem*{defn}{Definition}

\DeclareMathOperator{\Biggg}{Big}

\newcommand\categ[1]{\mathbf{#1}}
\newcommand\sets{\categ{Sets}}

\title{A Categorical Construction of Ultrafilters}
\author[D.~Litt]{Daniel Litt} 
\address{Department of Mathematics, Harvard University}
\email{dalitt@fas.harvard.edu}
\author[Z.~Abel]{Zachary Abel}
\address{Department of Mathematics, Harvard University}
\email{zabel@fas.harvard.edu}
\author[S.~D.~Kominers]{Scott~D.\ Kominers}
\address{Department of Mathematics, Harvard University\newline\indent c/o 8520 Burning Tree Road, Bethesda, MD, 20817}
\email{kominers@fas.harvard.edu}
\email{skominers@gmail.com}


\begin{document}\maketitle
\begin{center}
\textit{Dedicated to Professor Elem\'er Elad Rosinger}
\end{center}
\begin{abstract}
Ultrafilters are useful mathematical objects having applications
in nonstandard analysis, Ramsey theory, Boolean algebra, topology, and other
areas of mathematics. In this note, we provide a categorical construction of
ultrafilters in terms of the inverse limit of an inverse family of finite partitions;
this is an elementary and intuitive presentation of a consequence of the
profiniteness of Stone spaces. We then apply this construction to answer a
question of Rosinger in the negative.
\end{abstract}

\section{Introduction}

It is well-known that the category $\categ{Stone}$ of Stone spaces with continuous maps is categorically equivalent to the pro-completion of the category $\categ{FinSet}$ of finite sets (see~\cite[p.~236]{johnstone}).  We illuminate this equivalence in the context of spaces of ultrafilters, in an elementary setting which does not require topological methods.  In particular,
we give an elementary construction of ultrafilter spaces as an inverse limit,
without resorting to Stone spaces or to the correspondence between maximal ideals
and ultrafilters. We then give a brief application of this construction, answering a
question of Rosinger \cite{rosinger} in the negative.
\section{Ultrafilters}\label{sec:defn}
\begin{defn}\label{def:uf}
Let $S$ be a set.  An \term{ultrafilter on $S$} is a subset $\mathcal{U}$ of $2^S$, the power set of $S$, such that:  
\begin{enumerate}
\item $\emptyset\not\in \mathcal{U}$,
\item $A\in \mathcal{U}, A\subset B\implies B\in \mathcal{U}$,
\item $A\in \mathcal{U}, B\in \mathcal{U}\implies A\cap B\in \mathcal{U}$,
\item $A\not\in \mathcal{U}\implies S\setminus A\in \mathcal{U} $.
\end{enumerate}
\end{defn}
We say that an ultrafilter is \term{free} if it contains no finite sets.  It has been shown (see, for example, \cite[Form 63]{hr}, \cite[pp.~145--146]{pinto}) that
\begin{thm}[Free Ultrafilter Theorem]\label{thm:fut}
If $S$ is infinite, then there exists a free ultrafilter on $S$.
\end{thm}
The standard proof of Theorem~\ref{thm:fut}, given in \cite[pp.~145--146]{pinto}, considers, more generally, \term{filters} on $S$, i.e.\ subsets of $2^S$ satisfying Properties (1)--(3) from Definition~\ref{def:uf}.  It proves via Zorn's  lemma that given any filter $\mathcal{F}$, there exists an ultrafilter $\mathcal{U}\supset \mathcal{F}$.  Taking $\mathcal{F}$ to be the \term{cofinite filter} (the collection of all sets whose complements are finite) gives Theorem~\ref{thm:fut}.

Now, let $\sets$ denote the category of sets and let $\mathcal{FP}(S)\subset \sets$ denote the set of finite partitions of a set $S$.  Let $\mathcal{FPS}(S)\subset \sets$ denote the  set of finite partitions of \emph{subsets} of $S$ with the partial ordering defined as follows:   $\Delta'\leq \Delta$ if and only if for all $D'\in \Delta',$ there exists a unique $D\in \Delta$ such that $D'\subset D$, i.e.\ $\Delta'$ is a subset of a (possibly trivial) refinement of $\Delta$.  This turns $\mathcal{FPS}(S)$ into an inverse family with morphisms$$\{\psi_{\Delta', \Delta}: \Delta'\leq \Delta\},$$
where $\psi_{\Delta', \Delta}$ is defined by $$\psi_{\Delta', \Delta} :A\in \Delta'\mapsto B\in \Delta \text{ s.t. } A\subset B.$$

The following property of ultrafilters will be useful:
\begin{lem}\label{lemma:fpuf}
Let $\mathcal{U}$ be an ultrafilter on $S$, and let $\Delta\in \mathcal{FP}(S)$.  Then there exists a unique $D\in \Delta$ such that $D\in \mathcal{U}$.
\end{lem}
\begin{proof}
Assume to the contrary that no such $D$ exists.  Then $S\setminus D\in \mathcal{U}$ for each $D\in \Delta$.  Hence their intersection, $$\bigcap_{D\in \Delta} S\setminus D=\emptyset,$$ is in $\mathcal{U}$ by Property 3 of Definition~\ref{def:uf}, which contradicts Property 1 of Definition~\ref{def:uf}; that is, the empty set cannot be in $\mathcal{U}$.

Now, assume that $D, D'\in \Delta$ are both in $\mathcal{U}$.  Then $D\cap D'=\emptyset\in U$---again, a contradiction.
\end{proof}

\section{The Inverse Limit}\label{sec:invlim}
We require one additional definition, which is central to our categorical approach to ultrafilters:\begin{defn}The \term{inverse limit} of an inverse family $(X_i, f_{ij})$ in a category $\mathcal{C}$ is the universal object $X$ (unique up to a unique isomorphism) equipped with arrows $\pi_i: X\to X_i$ with $\pi_j=f_{ij}\circ \pi_i$.  That is, $X$ is such that for any $Y\in \mathrm{Ob}(\mathcal{C})$ and collection of maps $u_i: Y\to X_i$ such that $u_j=f_{ij}\circ u_i$ for all $f_{ij}$, there exists a unique $u: Y\to X$ such that  the diagram 
\begin{equation*}
\xymatrix@=50pt{
& Y \ar@{-->}[d]^{u} \ar[ddl]_{u_i} \ar[ddr]^{u_j} &\\
& X \ar[dl]^{\pi_i} \ar[dr]_{\pi_j} &\\
X_i \ar[rr]^{f_{ij}} && X_j
}
\end{equation*} commutes for all $f_{ij}$.\end{defn}
For an inverse family $(X_i, f_{ij})$ in $\sets$, the inverse limit can be explicitly constructed as
\begin{equation}\label{invlim}\varprojlim X_i =\left\{ (a_i)\in \prod X_i : a_j=f_{ij}(a_i) \text{ for all } f_{ij}\right\},\end{equation}
which may, in some cases, be empty.

\section{Our Categorical Construction}\label{catconst}
We may now give categorical interpretations of both the set free ultrafilters and the set of all ultrafilters over a set $I$.  In particular, consider the function $\Biggg: \mathcal{FP}(I)\to \mathcal{FPS}(I)$ given by $$\Biggg: \Delta \mapsto \{D\in \Delta: D \text{ is infinite}\}.$$
Then we have the following theorem:
\begin{thm} The set of free ultrafilters on $I$ is in canonical bijection with $$\varprojlim_{\Delta\in \mathcal{FP}(I)} \Biggg(\Delta).$$
Furthermore, the set of all ultrafilters on $I$ is in canonical bijection with $$\varprojlim_{\Delta\in \mathcal{FP}(I)} \Delta.$$
\end{thm}
\begin{proof}
We prove the second claim above; the first follows analogously.  We claim that each ultrafilter induces a unique element of the inverse limit by the mapping 
$$\Phi:  \mathcal{U}\mapsto (\Delta\cap \mathcal{U})_{\Delta\in \mathcal{FP}(I)}\in \varprojlim_{\Delta\in \mathcal{FP}(I)} \Delta,$$
where
$$ \varprojlim_{\Delta\in \mathcal{FP}(I)} \Delta\subseteqq  \prod_{\Delta\in \mathcal{FP}(I)} \Delta.$$  
We first check that any element of the image of the above map is in the inverse limit, as claimed.  First, note that for all $\Delta\in \mathcal{FP}(I)$, $\Delta\cap \mathcal{U}$ is a singleton by Lemma~\ref{lemma:fpuf}, so $\Phi(\mathcal{U})$ is indeed an element of $\prod_{\Delta\in \mathcal{FP}(I)} \Delta$.  To see that $\Phi(\mathcal{U})$ is in $\varprojlim_{\Delta\in \mathcal{FP}(I)} \Delta$, we check that $\Phi(\mathcal{U})$ satisfies the conditions of the construction in Equation \eqref{invlim} of Section~\ref{sec:invlim}.  In particular, we have that for all $\psi_{\Delta', \Delta}$ with $\Delta'\leq \Delta$, $$\mathcal{U}\cap \Delta'\subseteq\psi_{\Delta', \Delta}(\mathcal{U}\cap \Delta'),$$ so $\psi_{\Delta', \Delta}(\mathcal{U}\cap \Delta')\in \mathcal{U}$ by Property 2 of Definition~\ref{def:uf}.  But, by definition, $\psi_{\Delta', \Delta}(\mathcal{U}\cap \Delta')\in \Delta$, so $\mathcal{U}\cap \Delta= \psi_{\Delta', \Delta}(\mathcal{U}\cap \Delta')$, as desired (as each set contains a single element).  So we have that $$\Phi(\mathcal{U})\in \varprojlim_{\Delta\in \mathcal{FP}(I)} \Delta.$$

We claim that $\Phi$ is the desired canonical bijection.  To see that this map is injective, consider ultrafilters $\mathcal{U}, \mathcal{U}'$ with $\Phi(\mathcal{U})=\Phi(\mathcal{U}')$.  Note that for each $A\in \mathcal{U}$, we may take the partition $\Delta_A=\{A, I\setminus A\}$; then, as $\Phi(\mathcal{U})=\Phi(\mathcal{U}')$, we must have $\mathcal{U}'\cap \Delta_A=A$.  Thus, $A\in \mathcal{U}'$, so $\mathcal{U}\subseteq\mathcal{U}'$.  The reverse inclusion follows identically, so $\mathcal{U}=\mathcal{U}'$.

We now show that $\Phi$ is surjective.  Choose a tuple $(a_\Delta)\in \varprojlim \Delta$; we claim that the set $$U= \{a_\Delta: \Delta\in \mathcal{FP}(I)\}$$ is an ultrafilter and that $\Phi(U)=(a_\Delta)$.  To check that $U$ is an ultrafilter, we verify the four definitional properties.  
\begin{enumerate}
\item $\emptyset\not\in{U}$:  The empty set is not an element of any partition $\Delta$.
\item $A\in{U}, A\subset B\implies B\in{U}$:  Consider the partitions $\Delta_1=\{B, I\setminus B\}$ and $\Delta_2=\{A, B\setminus A, I\setminus B\}$.  Noting that $A=a_{\Delta_2}$, we have $\Delta_2\leq \Delta_1$ and thus $a_{\Delta_1}=\psi_{\Delta_2, \Delta_1}(A)=B$.  So $B\in U$, as desired.
\item $A\in{U}, B\in{U}\implies A\cap B\in{U}$:  Consider the partions $\Delta_1=\{A, I\setminus A\}, \Delta_2=\{B, I\setminus B\}, \Delta_3=\{A\cap B, A\setminus B, B\setminus A, I\setminus (A\cup B)\}$.  We have that $\Delta_3\leq \Delta_1, \Delta_2$, so $\psi_{\Delta_3, \Delta_1}(a_{\Delta_3})=A, \psi_{\Delta_3, \Delta_2}(a_{\Delta_3})=B$.  But then  $a_{\Delta_3}=A\cap B$, so $A\cap B\in U$.
\item $A\not\in{U}\implies I\setminus A\in{U}$:  Let $\Delta=\{A, I\setminus A\}$.  Then at least one of $A, I\setminus A$ (namely, $a_\Delta$) is in $U$.  Since it is not $A$ by assumption, it must be $I\setminus A$.
\end{enumerate}
Clearly, $\Phi(U)=(a_\Delta)$, by construction, so $\Phi$ is bijective.

An identical proof gives the first claim, as we never use the cardinality of the sets involved.  That is, the restriction by $\Biggg$ guarantees that all the elements of each $(a_\Delta)$ are infinite; there is always at least one infinite element in any finite partition of an infinite set (on finite sets, the inverse limit will indeed be empty, as there are no free ultrafilters on finite sets), by the pigeonhole principle.
\end{proof}

\section{A Concrete Example}
As an application of our results on ultrafilters, we note an interesting corollary:
\begin{thm}\label{thm:mt} Consider a function $f: I\to X$, where $I$ is an infinite indexing set.  For $\Delta\in \mathcal{FP}(X)$, let $\Delta(f)$ denote the set $$\Delta(f):=\{D\in \Delta: f^{-1}(D) \text{ is infinite}\}.$$  Then
$$\varprojlim_{\Delta\in \mathcal{FP}} \Delta(f)\not=\emptyset.$$
\end{thm}
\begin{proof}
For $\Delta\in \mathcal{FP}(X)$, let $$f^{-1}(\Delta)=\{f^{-1}(D): D\in \Delta\}.$$  Note that every partition in $\mathcal{FP}(I)$ admits a representation in this fashion.  Then any free ultrafilter $\mathcal{U}$ on $I$ gives an element of the inverse limit above, e.g. $$(D\in \Delta: f^{-1}(D)\in \mathcal{U})_{\Delta\in \mathcal{FP}(X)},$$ which is an element of the inverse limit precisely by the argument in Section~\ref{catconst}, above.
\end{proof}
\begin{cor}\label{thecor}
Let $X$ be a set and $T:X\to X$ be a function.  For $\Delta\in \mathcal{FP}(X)$, let $$\Delta(x):=\{D\in \Delta: \{n\in \mathbb{N}: T^n(x)\in D\} \text{ is infinite}\}.$$
Then for each $x\in X$, we have $$\varprojlim_{\Delta\in \mathcal{FP}(X)} \Delta(x)\not=\emptyset.$$
\end{cor}
\begin{proof}
Fixing $x$, we may take $I=\mathbb{N}$ and $f: n\mapsto T^n(x)$ in Theorem~\ref{thm:mt}.   The result follows immediately.
\end{proof}

Corollary~\ref{thecor} negatively answers the conjecture Rosinger posed in \cite{rosinger}.

\section*{Acknowledgements}
The authors are extremely grateful to Professor Elem\'er Elad Rosinger for bringing their attention to the problem and for his helpful comments and suggestions on the work.  They would also like to acknowledge Brett Harrison for his excellent suggestions on earlier drafts of this paper.

\end{document}